%% file: RAEVIa_corr.tex
\documentclass[a4paper,11pt]{article}

\usepackage{amssymb}
\usepackage{graphicx}
\usepackage{wrapfig}
\usepackage{caption}
\usepackage{subcaption}

\input{macro_a_RAEVI}

\begin{document}

\thispagestyle{empty}
\begin{center}
{\bf \LARGE
The resurgent character of the Fatou coordinates\\[1.5ex] of a simple
  parabolic germ}\\[3ex]

Artem Dudko, David Sauzin
\end{center}
\begin{abstract}
  Given a holomorphic germ at the origin of~$\C$ with a simple
  parabolic fixed point, the local dynamics is classically described
  by means of pairs of attracting and repelling Fatou coordinates
  and the corresponding pairs of horn maps, of crucial importance for
  \'Ecalle-Voronin's classification result and the definition of the
  parabolic renormalization operator.
  We revisit \'Ecalle's approach to the construction of Fatou
  coordinates, which relies on Borel-Laplace summation, and give an original
  and self-contained proof of their resurgent character.
\medskip

\end{abstract}

\section{Resurgence and summability of the formal Fatou coordinate}

Let us give ourselves a simple parabolic germ of holomorphic map at the
origin of~$\C$, \ie of the form
\beglabel{eqdefF}
F(w) = w + c w^2 + O(w^3) \in w\C\{w\}
\elabel
with $c \in \C^*$.
We rewrite it at~$\infty$ by means of the change of coordinate $z =
-1/ c w$:
\beglabel{eqdeff}
f(z) \defeq -\frac{1}{c F\big(-\frac{1}{c z}\big)} =
z + 1 + a(z),
\qquad
a(z) = -\rho z\ii + O(z^{-2}) \in z\ii\C\{z\ii\}
\elabel
with some $\rho\in\C$.
For every $R>0$ and $\de\in(0,\pi/2)$, we introduce the notations
\begin{align*}
\gP^+_{R,\de} & \defeq \big\{
r\,\ee^{\I\th} \in \C \mid r>R, \; -\tfrac{\pi}{2}-\de < \th < \tfrac{\pi}{2}+\de
\big\},\\[1ex]
\gP^-_{R,\de} & \defeq \big\{
r\,\ee^{\I\th} \in \C \mid r>R, \; \tfrac{\pi}{2}-\de < \th < \tfrac{3\pi}{2}+\de
\big\}.
\end{align*}
%


\begin{Def}
A \emph{pair of Fatou coordinates at~$\infty$} is a pair $(v^+,v^-)$ of
univalent maps
\[
v^+ \col \gP^+_{R,\de} \to \C,
\qquad
v^- \col \gP^-_{R,\de} \to \C,
\]
for some $R$ and~$\de$, such that
\beglabel{eqvpm}
v^+ \circ f = v^+ +1,
\qquad
v^- \circ f = v^- +1.
\elabel
\end{Def}


It is well-known that pairs of Fatou coordinates exist and are unique
up to the addition of a free pair of constants
$(c^+,c^-) \in \C^2$.
We shall see how to recover these facts by means of Borel-Laplace summation.


\begin{Nota}
Let $b(z) \defeq a(z-1)$ and
$b_1(z) \defeq b(z) + \rho\log\frac{1+z\ii b(z)}{1-z\ii}$, so that
$b(z) \in z\ii\C\{z\ii\}$ and $b_1(z) \in z^{-2}\C\{z\ii\}$.
We denote by $C_{\id-1}$ and $C_{\id+b}$ the composition operators
$\ph(z) \mapsto \ph(z-1)$ and
$\ph(z) \mapsto \ph\big(z+b(z)\big)$,
acting in spaces of functions as well as in the space of formal series $\C[[z\ii]]$.
\end{Nota}


Suppose that $\gP = \gP^\pm_{R,\de}$ with $R\,$ large enough so that~$b$
is analytic on~$\gP$ and let $\Log$ be any branch of the logarithm
in~$\gP$.
Under the change of unknown function $v(z) = z + \rho\Log z + \ph(z)$,
the equation $v\circ f = v+1$ is transformed\footnote{
Notice that, for $\abs{z}$ large enough, $\frac{f(z)}{z}$ is close to~$1$,
hence
$\Log f(z) - \Log z = \log \frac{f(z)}{z}$ where $\log$ is the
principal branch of the logarithm.
} into
\beglabel{eqph}
C_{\id-1} \ph = C_{\id+b} \ph + b_1.
\elabel
%


\begin{Th}     \label{ThmResur}
Equation~\eqref{eqph} admits a unique formal solution of the form
$\ti\ph(z) = \sum_{n\geq0} c_n z^{-n-1}$.
Its formal Borel transform
$\hat\ph \defeq \gB\ti\ph = \sum_{n\geq0} c_n \frac{\ze^n}{n!}$
is $2\pi\I\Z$-resurgent, in the sense that it
converges for $\abs{\ze} < 2\pi$ and extends analytically along any
path issuing from~$0$ and staying in $\C\setminus 2\pi\I\Z$ except for
its origin.

Moreover, for such a path~$\ga$ with endpoint $\ze_* \in \I\R$ or
for $\ga=\{0\}$ and $\ze_*=0$,
and for $\de_0 \in (0,\frac{\pi}{2})$, there exist $C_0,R_0>0$ such that
the analytic continuation $\cont_\ga\hat\ph$ of~$\hat\ph$ along~$\ga$
(which is a holomorphic germ at~$\ze_*$)
satisfies
\beglabel{ineqleftright}
\abs*{ \cont_\ga\hat\ph\big( \ze_* + t\,\ee^{\I\th} \big) }
\leq
C_0\,\ee^{R_0 t}
\quad \text{for all $t\geq0$ and $\th \in I_{\de_0}^+\cup I_{\de_0}^-$,}
\elabel
where $I_{\de_0}^+ \defeq [-\de_0,\de_0]$ and $I_{\de_0}^-\defeq [\pi-\de_0,\pi+\de_0]$.
\end{Th}
\noindent
In the case $\ga=\{0\}$, inequality~\eqref{ineqleftright} says that
the formal series~$\ti\ph$ is Borel-Laplace $1$-summable in the
directions of~$I_{\de_0}^+$ and~$I_{\de_0}^-$. 
In fact, Theorem~\ref{ThmResur} says much more, since it provides a
Riemann surface to which~$\hat\ph$ extends, with exponential estimates
on all its sheets.
The rest of this section is devoted to the proof of
Theorem~\ref{ThmResur}.


The operator $C_{\id-1}-\ID$ maps $\zcz$ to $z^{-2}\C[[z\ii]]$
bijectively, we denote by~$E$ the inverse, which decreases the order
by one unit and has for Borel counterpart the multiplication operator
\[
\hat E \col \hat\psi(\ze) \in \ze\C[[\ze]] \mapsto
\frac{1}{\ee^\ze-1} \hat\psi(\ze) \in \C[[\ze]].
\]
On the other hand, $B \defeq C_{\id+b}-\ID =
\sum_{k\geq1} \frac{1}{k!} b^k \big(\frac{\dd}{\dd z}\big)^k$
(formally convergent series of operators)
increases the order by at least two units, so that the
composition~$EB$ maps $z^{-k-1}\C[[z\ii]]$ in
$z^{-k-2}\C[[z\ii]]$ for every $k\geq0$.


\begin{Lm}   \label{Lemformalphk}
Let $\ti\ph_k \defeq ( E B )^k E b_1 \in z^{-k-1} \C[[z\ii]]$ for
$k\geq0$.
Then the formally convergent series $\sum_{k\geq0}
\ti\ph_k$ is the unique solution of~\eqref{eqph} in $\zcz$.
\end{Lm}

\begin{proof}
For $\ti\ph \in \zcz$, \eqref{eqph} $\Leftrightarrow
(C_{\id-1}-\ID)\ti\ph = B\ti\ph + b_1 \Leftrightarrow
\ti\ph = EB\ti\ph + E b_1$,
whence the conclusion follows.
\end{proof}

\begin{figure}
\centering
\begin{subfigure}{0.35\textwidth}
\centering
\includegraphics[scale=0.17]{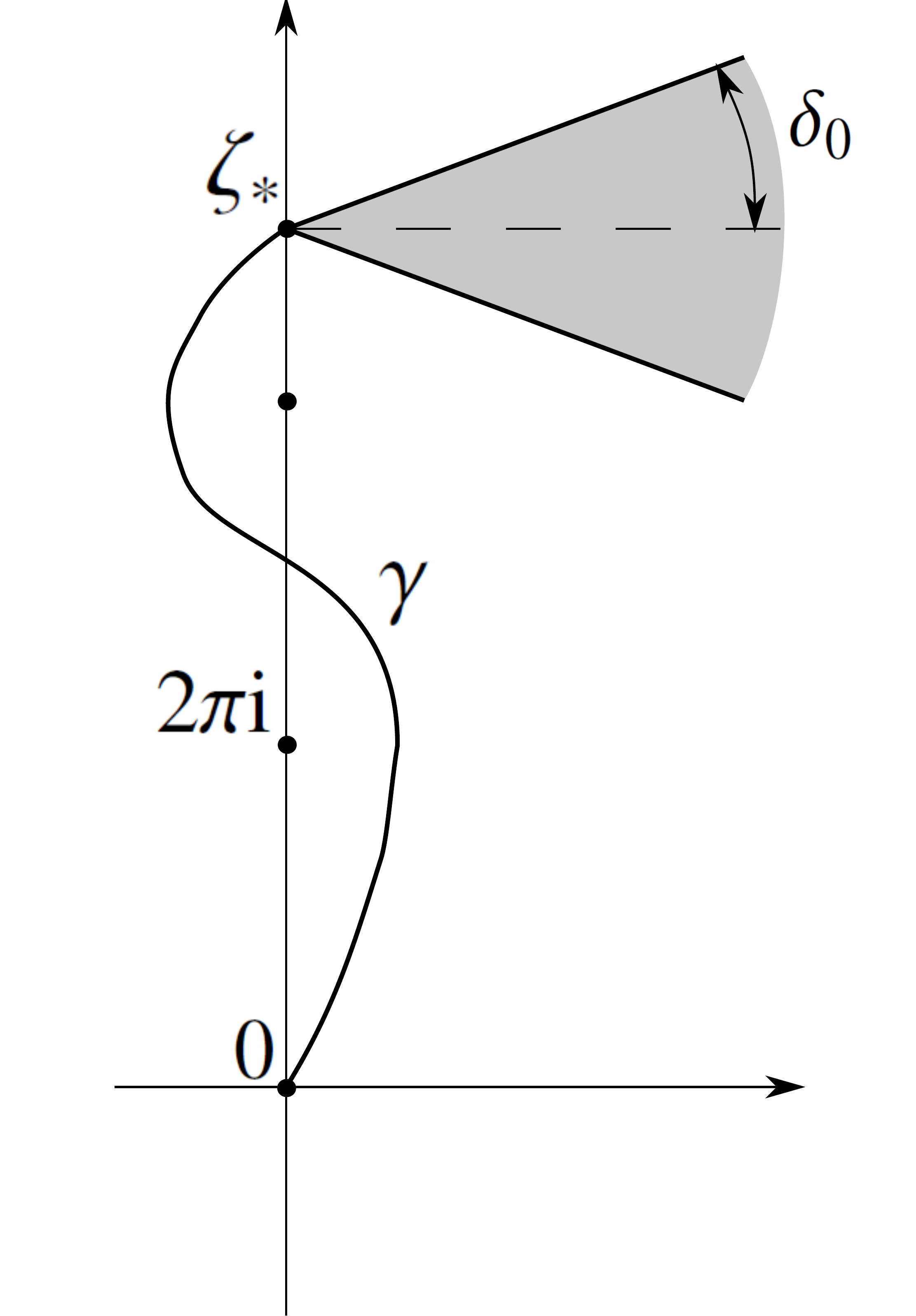}
\caption{Illustration of Theorem~\ref{ThmResur}.}
\end{subfigure}
\hspace{1cm}
\begin{subfigure}{0.35\textwidth}
\includegraphics[scale=0.10]{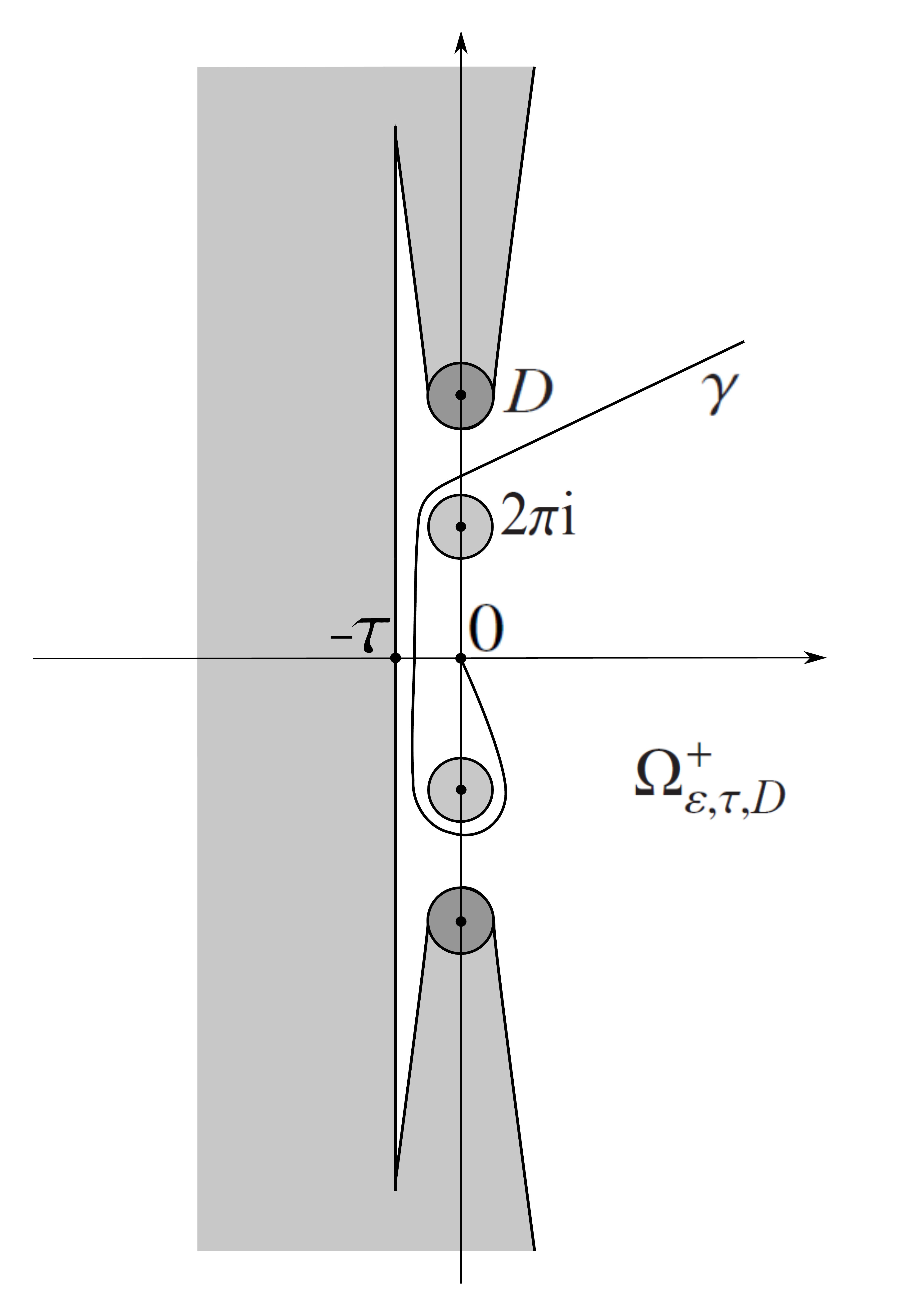}
\caption{Illustration of Lemma~\ref{Lemcontphk}.}
\end{subfigure}
\end{figure}

\begin{Lm}   \label{Lemcontphk}
For each $k\geq0$ the formal Borel transform $\hat\ph_k \defeq \gB\ti\ph_k$ is
$2\pi\I\Z$-resurgent.

Suppose $0 < \eps < \pi < \tau$, $0 < \ka \leq1$ and $D$ is a closed
disc of radius~$\eps$ centred at $2\pi\I m_0$ with $m_0 \in \Z^*$.
Then there exist $C_1,M,R>0$ such that,
for every piecewise $C^1$ path with natural parametrisation
$\ga \col [0,\ell] \to \C$ satisfying
%
%
\begin{gather}
\label{condeps}
s\in[0,\eps] \ens\Rightarrow\ens \abs{\ga(s)} = s,
\qquad
s>\eps \ens\Rightarrow\ens \abs{\ga(s)} > \eps,\\[1ex]
\label{condtauka}
\qquad
s\in[0,\ell] \quad\Rightarrow\quad
\abs{\ga(s)} > \ka s
\ens\text{and}\ens
\ga(s) \in \Om^+_{\eps,\tau,D},
\end{gather}
where
$\Om^+_{\eps,\tau,D} \defeq \{\, \ze\in\C \mid
\RE \ze > -\tau, \;
\dist\big( \ze, 2\pi\I\Z^* \big) > \eps \,\}
\setminus
\{\, u\ze \in\C \mid u\in[1,+\infty), \;
\pm\ze \in D \,\}$,
%
%
one has
\beglabel{ineqestimphk}
\abs*{ \cont_\ga\hat\ph_k\big( \ga(\ell) \big) }
\leq C_1\frac{ (M\ell)^k }{k!} \ee^{R\ell}
\quad \text{for all $k\geq0$}.
\elabel
\end{Lm}


\begin{proof}
Choose $M_0, \la >0$ so that
$\abs*{ \frac{\ze}{\ee^\ze-1} } \leq M_0 \, \ee^{-\la \abs{\ze}}$
for $\ze \in \Om^+_{\eps,\tau,D} $.
%
%
%
%
%
We have $\hat b\defeq\gB b$, $\hat b_1 \defeq \gB b_1$ entire
functions, with
\[
\abs{\hat b(\ze)} \leq C \,\ee^{\be\abs{\ze}}
\ens\text{and}\ens
\abs{\hat b_1(\ze)} \leq C \abs{\ze} \,\ee^{\be\abs{\ze}}
\quad\text{for all $\ze\in\C$,}
\]
for some $C,\be>0$.
The formal Borel transform maps $b^k(z)$ to $\hat b^{*k}(\ze)$, the
$k$-th convolution\footnote{
The formal Borel transform~$\gB$ maps the Cauchy product of formal
series to the convolution product defined by
$\hat\phi*\hat\psi(\ze) = \int_0^\ze \hat\phi(\ze-\xi)\hat\psi(\xi)\,\dd\xi$,
with termwise integration for formal series
$\hat\phi, \hat\psi \in \C[[\ze]]$, and with the
obvious analytical meaning when $\hat\phi, \hat\psi \in \C\{\ze\}$ (then
taking~$\ze$ close enough to~$0$).
%
%
}
power of $\hat b(\ze)$, which is also entire and
for which we get
\beglabel{ineqbconvk}
\abs{\hat b^{*k}(\ze)} \leq C^k \frac{\abs{\ze}^{k-1}}{(k-1)!} \ee^{\be\abs{\ze}}
\quad\text{for all $\ze\in\C$.}
\elabel
The Borel counterpart of $C_{\id+b}-\ID$ is
$\hat B \col \hat\psi \in \C[[\ze]] \mapsto
\sum_{k\geq1} \frac{1}{k!} \hat b^{*k} * ( (-\ze)^k \hat\psi) \in \C[[\ze]]$,
which leaves invariant the subspace $\C\{\ze\}$ and induces on it an integral
transform
\[
\hat B\hat\psi(\ze) = \int_0^\ze K(\xi,\ze) \hat\psi(\xi) \,\dd\xi
\quad\text{for $\ze$ close to~$0$,}
\]
with kernel function
$K(\xi,\ze) = \sum_{k\geq1} \frac{(-\xi)^k}{k!} \hat b^{*k}(\ze-\xi)$.
The point is that this kernel is holomorphic in $\C\times\C$, with
estimates following from~\eqref{ineqbconvk}:
\[
\abs{K(\xi,\ze)} \leq C \abs{\xi} \, \ee^{\mu\abs{\xi} + (\frac{C}{\mu}+\be)\abs{\ze-\xi}}
\quad\text{for all $(\xi,\ze)\in\C\times\C$,}
\]
with any $\mu>0$.
Therefore, whenever a germ $\hat\psi\in\C\{\ze\}$ admits
analytic continuation along a naturally parametrised path $\ga \col [0,\ell] \to \C$, this is
also true for $\hat B\hat\psi$, with
$\cont_\ga \hat B\hat\psi\big( \ga(s) \big) =
\int_0^s K\big( \ga(\sig),\ga(s)\big) \cont_\ga \hat\psi\big(\ga(\sig) \big)
\ga'(\sig)\,\dd\sig$
for all $s\in[0,\ell]$.
Hence $2\pi\I\Z$-resurgence is preserved by the composition $\hat
E\hat B$ and we can get estimates.


We start with
$\hat\ph_0 = \hat E\hat b_1 = \frac{\hat b_1(\ze)}{\ee^\ze-1}$,
which is meromorphic and clearly $2\pi\I\Z$-resurgent,
and satisfies~\eqref{ineqestimphk} with $C_1 \defeq M_0 C$ and with any $R\geq\be$.
We choose $\mu \defeq \la\ka$, $R\defeq \frac{C}{\mu} + \be$ and
$M \defeq \frac{M_0 C}{\ka}$.
We obtain the desired result for $\hat\ph_k = (\hat E\hat B)^k \hat\ph_0$ by induction,
observing that if a path~$\ga$ satisfies~\eqref{condeps}--\eqref{condtauka}
and a $2\pi\I\Z$-resurgent function~$\hat\psi$ satisfies
$\abs*{\cont_\ga \hat\psi\big( \ga(s) \big)} \le \ee^{R s} \Psi(s)$ for
all $s\in[0,\ell]$ then,
for all $0\le\sig\le s\le \ell$,
$\abs{K(\ga(\sig),\ga(s))} \le C\sig \ee^{\mu\sig+R(s-\sig)}$ 
and 
$\abs*{ \frac{1}{\ee^{\ga(s)}-1} } \leq \frac{M_0}{\ka s} \, \ee^{-\mu s}$, whence
\[
\abs*{\cont_\ga \hat E\hat B\hat\psi\big( \ga(s) \big)} \le
\tfrac{C M_0}{\ka} \, \ee^{Rs} \int_0^s \tfrac{\sig}{s} 
\Psi(\sig) \, \ee^{-\mu(s-\sig)}\,\dd\sig \le
M\,\ee^{R s} \int_0^s \Psi(\sig)\,\dd\sig
\quad \text{for all $s\in[0,\ell]$.}
\]
\vskip -12pt
\end{proof}
\vspace{1ex}

We can deduce that the series of holomorphic functions
$\sum_{k\geq0} \hat\ph_k$ converges normally in any compact subset of
the disc $\big\{ \abs{\ze} < 2\pi \big\}$
(using paths~$\ga$ of the form $[0,\ze]$)
and that its sum, which is~$\hat\ph$, extends analytically along any
naturally parametrised path $\ga$ which starts as the line segment
$[0,1]$ and then stays in $\C\setminus 2\pi\I\Z$:
indeed, taking $\eps,\ka$ small enough and $\tau, m_0$ large enough,
we see that Lemma~\ref{Lemcontphk} applies to~$\ga$ and the neighbouring paths,
so that \eqref{ineqestimphk} yields the normal convergence of
$\sum_{k\geq0} \cont_\ga\hat\ph_k(\ga(t)+\ze) = \cont_\ga\hat\ph(\ga(t)+\ze)$
for all~$t$ and~$\ze$ with $\abs{\ze}$ small enough.
Therefore $\hat\ph$ is $2\pi\I\Z$-resurgent.


We also get the part of~\eqref{ineqleftright} relative to~$I_{\de_0}^+$:
given $\de_0\in(0,\pi/2)$ and~$\ga$ with endpoint $\ze_*\in\I\R$, we first
replace an initial portion of~$\ga$ with a line segment of length~$1$
(unless~$\ga$ stays in the unit disc, in which case the modification
of the arguments which follow is trivial)
and switch to its natural parametrisation
$\ga \col [0,\ell] \to \C$.
We then choose
\begin{align*}
\eps & < \min\bigg\{ 1, \, \min_{[1,\ell]}\abs{\ga}, \,
\dist\big( \ga\big([0,\ell]\big), 2\pi\I\Z^* \big), \,
\dist\big(\ze_*,2\pi\I\Z\big)\cos\de_0 \bigg\}, \\[1ex]
\ka & < \min\bigg\{
\min_{[0,\ell]} \tfrac{\abs{\ga(s)}}{s}, \,
\min_{t\geq0} \tfrac{\abs{ \ze_* + t\,\ee^{\pm\I\de_0} }}{\ell+t}
\bigg\},
\end{align*}
$\tau > -\min\RE\ga$ and $m_0 > \tfrac{1}{2\pi}(\eps + \max
\abs{\IM\ga})$,
so that, for each $t\geq0$ and $\th\in I^+_{\de_0}$,
Lemma~\ref{Lemcontphk} applies to the concatenation $\Ga
\defeq \ga + [\ze_*,\ze_*+t\,\ee^{\I\th}]$;
since $\Ga$ has length $\ell+t$,
\eqref{ineqestimphk} yields
$ \abs*{ \cont_\ga\hat\ph\big( \ze_* + t\,\ee^{\I\th} \big) }
= \abs*{ \cont_\Ga\hat\ph\big( \Ga(\ell+t) \big) } \leq
C_1\,\ee^{(M+R)(\ell+t)} $.


The part of~\eqref{ineqleftright} relative to~$I_{\de_0}^-$ follows
from the fact that $\hat\ph^-(\ze) \defeq \hat\ph(-\ze)$ satisfies all
the properties we just obtained for $\hat\ph(\ze)$, since
it is the formal Borel transform of
$\ti\ph^-(z) \defeq -\ti\ph(-z)$
which solves the equation
$C_{\id-1} \ti\ph^- = C_{\id+b^-} \ti\ph^- + b^-_*$
associated with the simple parabolic germ $f^-(z) \defeq -f\ii(-z) = z+1+b^-(z+1)$.

The proof of Theorem~\ref{ThmResur} is now complete.


\begin{Rem}
The name ``resurgent'',
whose definition in a particular case is given in the statement of Theorem~\ref{ThmResur},
was chosen by \'Ecalle because of the explicit connection he obtained
between the germ~$\hat\ph$ at~$0$ and the singularities of its
analytic continuation at any point of $2\pi\I\Z$---see \cite{DSdeux}.
Theorem~\ref{ThmResur} is stated in \cite{Eca81}, with a detailed
proof only for the case $\rho=0$.
The above proof is original and simpler.
\end{Rem}


\begin{Rem}
For a path~$\ga$ ending at a point~$\ze$, 
we get the explicit formula
\[
\cont_\ga \hat\ph(\ze) = \frac{1}{\ee^\ze-1} \Big( \hat b_1(\ze) +
\sum_{k\ge1} \int_{\De_{\ga,k}} \hat b_1(\xi_1) \frac{
K(\xi_1,\xi_2)\cdots K(\xi_{k-1},\xi_k)K(\xi_k,\ze)
}{
(\ee^{\xi_1}-1) \cdots (\ee^{\xi_{k-1}}-1)(\ee^{\xi_k}-1)
}
\, \dd\xi_1 \wedge \cdots \wedge \dd\xi_k
\Big),
\]
with the notation
$\De_{\ga,k} \defeq \big\{\, \big( \ga(s_1),\ldots,\ga(s_k) \big) 
\mid s_1 \le \cdots \le s_k \,\big\}$ for each positive integer~$k$.
\end{Rem}

\section{Fatou coordinates and lifted horn maps}  \label{secFatouCoord}

%
For $\ga=\{0\}$ and any $\de_0 \in (0,{\pi}/{2})$, Theorem~\ref{ThmResur} provides $R_0 = R_0(\de_0)$
and we may apply the Laplace transform:
\[
\cL^\th \hat\ph(z) \defeq \int_0^{\ee^{\I\th}\infty} \ee^{-z\ze}
\hat\ph(\ze)\,\dd\ze,
\qquad \th \in I_{\de_0}^+ \cup I_{\de_0}^-,
\]
defining $\cL^\th \hat\ph$ holomorphic in
$\Pi_{R_0(\de_0)}^\th \defeq \{\, z\in\C \mid \RE(z\,\ee^{\I\th}) > R_0(\de_0)\,\}$.
The functions $\cL^\th\hat\ph$ with $\th\in I_{\de_0}^+$
are extension of each other, similarly for $I_{\de_0}^-$,
so we get two Borel-Laplace sums
\[
\cL^+\hat\ph \ens\text{in}\;
U^+ \defeq \bigcup_{\de_0\in(0,\frac{\pi}{2})} \bigcup_{\th\in I_{\de_0}^+} \Pi_{R_0(\de_0)}^\th,
\qquad \cL^-\hat\ph \ens\text{in}\;
U^- \defeq \bigcup_{\de_0\in(0,\frac{\pi}{2})} \bigcup_{\th\in I_{\de_0}^-} \Pi_{R_0(\de_0)}^\th.
\]
Since $U^\pm \subset \C\setminus\R^\mp$, we choose branches of the logarithm as follows:
\[
\Log^+ z = \int_1^z \frac{\dd u}{u}
\ens \text{for $z \in \C\setminus\R^-$},
\qquad
\Log^- z = \I\pi + \int_{-1}^z \frac{\dd u}{u}
\ens \text{for $z \in \C\setminus\R^+$}.
\]

\begin{wrapfigure}{r}{0.34\textwidth}
\vspace{-0.55cm}
\hspace{0.4cm}
\includegraphics[width=0.29\textwidth]{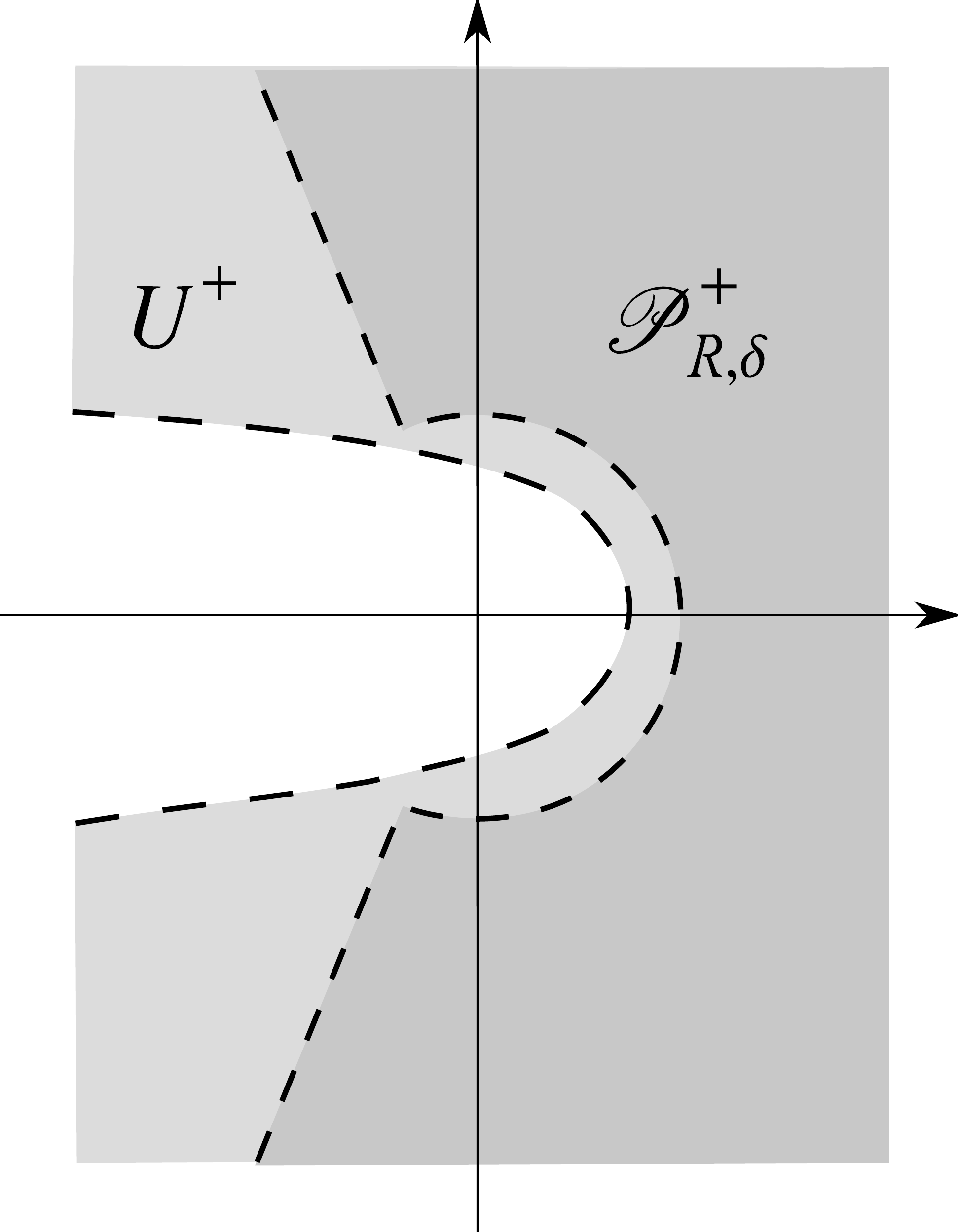}
\vspace{-0.15cm}
\caption*{(c) Illustration of Theorem \ref{CorFatou}.}
\vspace{-0.8cm}
\end{wrapfigure}
\begin{Th}       \label{CorFatou}
The formulas
\beglabel{eqdefvpm}
v_*^\pm(z) \defeq z + \rho\Log^\pm z + \cL^\pm\hat\ph(z),
\qquad z \in U^\pm,
\elabel
define a pair $(v_*^+,v_*^-)$ of Fatou coordinates at~$\infty$ for the
simple parabolic germ~$f$.
For any $0<\de'<\de<\pi/2$, provided that $R$ is large enough,
the function $\cL^\pm\hat\ph$ has uniform $1$-Gevrey asymptotic
expansion~$\ti\ph$ in~$\gP^\pm_{R,\de}$
and the function~$v_*^\pm$ is univalent in~$\gP^\pm_{R,\de}$ with its image
containing a set of the form $\gP^\pm_{R',\de'}$.

Any other pair of Fatou coordinates is of the form $(v_*^++c^+,v_*^-+c^-)$
with arbitrary complex constants $c^+,c^-$.
\end{Th}
\noindent
  (For the definition and properties of $1$-Gevrey asymptotic expansion, see
  \cite{Mal_cours} or \cite{Ram}.)


\begin{proof}
%
%
%
Since~$b_1$ is convergent, it coincides with $\cL^+\hat b_1$ on~$U^+$,
similarly $b_{|U^+} = \cL^+\hat b$.
We have $\ee^\ze\hat\ph = \hat\ph+\hat B\hat\ph + \hat b_1$
and $\cL^+\big(\ee^\ze\hat\ph\big)(z) = (\cL^+\hat\ph)(z-1)$.
The Laplace transform maps the convolution product
to the product of functions and the operator of multiplication
by~$-\ze$ to the differentiation~$\frac{\dd}{\dd z}$, 
hence its action on each term of
$\hat\ph + \hat B\hat\ph = \hat\ph + 
\sum_{k\geq1} \frac{1}{k!} \hat b^{*k} * ( (-\ze)^k \hat\ph )$ 
is clear and~\eqref{ineqbconvk} allows us to exchange~$\cL^+$ and~$\sum$:
\[
\cL^+(\hat\ph+\hat B\hat\ph) =
\sum_{k\geq0} \frac{1}{k!}b^k \big(\tfrac{\dd}{\dd z}\big)^k\cL^+\hat\ph =
\cL^+\hat\ph \circ (\id + b).
\]
It follows that $\cL^+\hat\ph$ is a solution of~\eqref{eqph}, and
hence that~$v_*^+$ satisfies~\eqref{eqvpm}.

Observe that $\gP^+_{R,\de} \subset U^+$ as soon as there
exists~$\de_0$ such that $\de < \de_0$ and
$R \geq R_0(\de_0) / \sin(\de_0-\de)$.
The $1$-Gevrey asymptotic property is standard in Borel-Laplace
summation. It implies $\cL^+\hat\ph = O(z\ii)$, $\frac{\dd}{\dd z}(\cL^+\hat\ph)
= O(z^{-2})$,
and also the univalence of~$v_*^+$ in $\gP^+_{R,\de}$ for $R$ large
enough and the statement on its image,
because $\abs*{\frac{v_*^+(z)}{z}} \xrightarrow{}1$
and $\frac{v_*^+(z)}{\abs*{v_*^+(z)}} \sim \frac{z}{\abs*{z}}$.

For any univalent solution~$v^+$ of~\eqref{eqvpm} on~$\gP^+_{R,\de}$,
the univalent function $v^+\circ (v_*^+)\ii \col \gP^+_{R',\de'} \to \C$
conjugates $\id+1$ with itself, hence is of the form $\id+P$ with a
$1$-periodic~$P$ and extends univalently to~$\C$,
hence~$P$ is constant. For $v_*^-$ the proof is similar.
\end{proof}


\begin{Rem}
The first use of Borel-Laplace summation for obtaining Fatou coordinates is in \cite[\S 9c]{Eca81}.
The asymptotic property without the Gevrey qualification can be found
in earlier works by G.~Birkhoff, G.~Szekeres, T.~Kimura and
J.~\'Ecalle---see \cite{Loray};
see \cite{LY} for a recent independent proof and an application to
numerical computations.
\end{Rem}


For any pair of Fatou coordinates $(v^+,v^-)$, we can take $\de'<\de$
and large $R,R'$ so that $\gP^+_{R,\de} \cap
(v^-)\ii(\gP^-_{R',\de'})$ has upper and lower connected components
$\gQ\up$ and~$\gQ\low$.
The \emph{lifted horn maps} are $h\up \defeq v^+ \circ (v^-)\ii_{| \gQ\up}$,
which conjugates $\id+1$ with itself, hence is of the form $\id+P\up$
with a $1$-periodic~$P\up$ and extends to an upper half-plane,
and $h\low \defeq v^+ \circ (v^-)\ii_{| \gQ\low}$ with similar
properties in a lower half-plane.
For the normalized pair $(v_*^+,v_*^-)$ we get
$h_*\uplow(Z) = Z + \rho(\Log^+ Z - \Log^- Z) + o(1)$,
where $\Log^+-\Log^- \equiv 0$ in~$\gQ\up$ and $\equiv-2\pi\I$
in~$\gQ\low$,
which puts a constraint on the form of the Fourier expansion of $h_*\uplow-\id$:
\beglabel{eqFourierHorn}
h_*\up(Z) - Z = \sum_{m\geq1} A_m \ee^{2\pi\I m Z},
\qquad
h_*\low(Z) - Z = -2\pi\I\rho + \sum_{m\geq1} A_{-m} \ee^{-2\pi\I m Z}.
\elabel

\begin{Def}
The coefficients~$A_m$, $m\in\Z^*$,
are called the \emph{\'Ecalle-Voronin invariants} of the simple parabolic germ~$f$.
\end{Def}


%
This name is motivated by \'Ecalle-Voronin's classification result
(\cite{Eca81}, \cite{Vo}, \cite{Mal_Bour}):
(i) two simple parabolic germs $f,g$ are analytically conjugate if and
only if there exists $c\in\C$ such that, for all~$m$,
$A_m(g) = \ee^{2\pi\I mc} A_m(f)$
(direct consequence of the results of this section);
(ii) any pair of Fourier series of the form
$\big( \sum_{m\geq1} A_m \ee^{2\pi\I mZ}, A_0 + \sum_{m\geq1} A_{-m} \ee^{-2\pi\I mZ} \big)$,
where the first (\resp second) one is holomorphic in an upper (\resp lower) half-plane,
can be obtained as
$( h_*\up-\id, h_*\low-\id )$
for a simple parabolic germ~$f$
(this part of the result is more difficult).


By definition, the \emph{horn maps} $H_*\uplow(w)$ are the lifted horn
maps $h_*\uplow$ represented in the coordinate $w=\ee^{2\pi\I Z}$,
\ie $H_*\uplow(w) \defeq \exp\big( 2\pi \I h_*\uplow(z))$, so that one
gets a germ at~$0$ and a germ at~$\infty$ on the Riemann sphere:
\[
H_*\up(w)=w + 2\pi\I A_1 w^2 + O(w^3) \in w\C\{w\},
\ens
H_*\low (w)=
\ee^{4\pi^2\rho} \big(w + 2\pi\I A_{-1} + O(w^{-1}) \big)\in w\C\{w\ii\}.
\]
Some authors also refer to them as the \'Ecalle-Voronin invariants.
The parabolic renormalization operator~$\gR$, of great importance in recent developments
of complex dynamics, is defined by $\gR \col F \mapsto H_*\up$
(recall that $F(w) = w + c w^2 + O(w^3) \in w\C\{w\}$ was assumed to
be a simple parabolic germ at~$0$, \ie $c\neq0$, and notice that $\gR
F$ is itself a parabolic germ at~$0$, but not necessarily simple:
this requires $A_1\neq 0$).

%
In this article, we have only exploited the information on the principal branch
of~$\hat\ph$ given by Theorem~\ref{ThmResur}.
The reader is referred to \cite{DSdeux} for an investigation of the
singularities of all the branches of~$\hat\ph$ and their relation to the
\'Ecalle-Voronin invariants,
which shows how the horn maps 
are encoded in the Borel plane.


\newpage

\noindent{\footnotesize \emph{Acknowledgements.}
  The authors are grateful to F.~Fauvet and M.~Yampolsky for fruitful
  discussions and encouragements.
  The first author expresses thanks to the Centro di Ricerca
  Matematica Ennio De Giorgi and Fibonacci Laboratory for
  hospitality.
  The second author acknowledges the support of the French National
  Research Agency under the reference ANR-12-BS01-0017.  }

\medskip



\vspace{1cm}

\noindent
Artem Dudko\\[1ex]
Institute for Mathematical Sciences \\
University of Stony Brook, NY, USA

\bigskip

\bigskip

\noindent
David Sauzin\\[1ex]
CNRS UMI 3483 - Laboratoire Fibonacci,\\
Centro di Ricerca Matematica Ennio De Giorgi,\\
Scuola Normale Superiore di Pisa, Italy


\end{document}

%% file: macro_a_RAEVI.tex
\usepackage{amsfonts}
\usepackage{amsmath}
\usepackage{amssymb}
\usepackage{color}
\usepackage{graphicx}
\usepackage{xspace}
\usepackage[matrix,arrow]{xy}
\usepackage{mathrsfs,eucal,enumerate}
\usepackage{amsthm,amscd}
\usepackage{mathtools}
\usepackage{epsfig,psfrag}
\usepackage{verbatim}
\usepackage{xr}

\newtheorem{Th}{Theorem}

\newtheorem{Lm}{Lemma}

\theoremstyle{definition}
\newtheorem{Def}{Definition}
\newtheorem{Rem}{Remark}

\newtheorem*{Nota}{Notation}

\renewcommand{\ge}{\geqslant}
\renewcommand{\le}{\leqslant}

\newcommand{\zcz}{z\ii\C[[z\ii]]}

\newcommand{\be}{\beta}
\newcommand{\de}{\delta}
\newcommand{\De}{\Delta}
\newcommand{\ga}{\gamma}
\newcommand{\Ga}{\Gamma}
\newcommand{\sig}{\sigma}
\renewcommand{\th}{\theta}

\newcommand{\ph}{\varphi}
\newcommand{\eps}{\varepsilon}
\newcommand{\ka}{\kappa}
\newcommand{\la}{\lambda}

\newcommand{\Om}{\Omega}
\newcommand{\ze}{\zeta}

\newcommand{\dd}{{\mathrm d}}      
\newcommand{\ee}{{\mathrm e}}      

\newcommand{\id}{\operatorname{id}}
\newcommand{\ID}{\operatorname{Id}}

\newcommand{\defeq}{\coloneqq} 
\newcommand{\col}{\colon\thinspace}          

\newcommand{\ii}{^{-1}}
\newcommand{\ti}{\tilde}
\newcommand{\ens}{\enspace}

\newcommand{\ie}{\emph{i.e.}\ }

\newcommand{\resp}{\emph{resp.}\ }

\DeclareMathOperator{\IM}{Im}         
 \DeclareMathOperator{\RE}{Re}        
\DeclareMathOperator{\Log}{Log}       
\DeclareMathOperator{\dist}{dist}     

\newcommand{\cont}{\operatorname{cont}}

\newcommand{\I}{\mathrm{i}}

\newcommand{\C}{\mathbb{C}}      

\newcommand{\R}{\mathbb{R}}        
\newcommand{\Z}{\mathbb{Z}}        

\newcommand{\cL}{ {\cal L }}

\newcommand{\gB}{\mathscr B}       
\newcommand{\gP}{\mathscr P}       
\newcommand{\gQ}{\mathscr Q}       
\newcommand{\gR}{\mathscr R}       

\newcommand{\beglabel}[1]{\begin{equation}	\label{#1}}
\newcommand{\elabel}{\end{equation}}

\DeclarePairedDelimiter\abs{\lvert}{\rvert}%

\newcommand{\up}{^\textrm{up}}
\newcommand{\low}{^\textrm{low}}
\newcommand{\uplow}{^\textrm{up/low}}

